\documentclass[12pt]{article} % use larger type; default would be 10pt

\usepackage[utf8]{inputenc} % set input encoding (not needed with XeLaTeX)

\usepackage{geometry} % to change the page dimensions
\usepackage{graphicx} % support the \includegraphics command and options

\usepackage[parfill]{parskip} % Activate to begin paragraphs with an empty line rather than an indent

\usepackage{booktabs} % for much better looking tables 
\usepackage{array} % for better arrays (eg matrices) in maths
\usepackage{paralist} % very flexible & customisable lists (eg. enumerate/itemize, etc.)
\usepackage{verbatim} % adds environment for commenting out blocks of text & for better verbatim
\usepackage{subfig} % make it possible to include more than one captioned figure/table in a single float
\usepackage{fancyhdr} % This should be set AFTER setting up the page geometry
\usepackage[nottoc,notlof,notlot]{tocbibind} % Put the bibliography in the ToC
\usepackage[titles,subfigure]{tocloft} % Alter the style of the Table of Contents

\usepackage[T1]{fontenc}
\usepackage[french]{babel}
\usepackage{amsmath,bm}
\usepackage{amsfonts}
\usepackage{amsthm}
\usepackage{enumitem}

\begingroup
    \makeatletter
    \@for\theoremstyle:=definition,remark,plain\do{%
        \expandafter\g@addto@macro\csname th@\theoremstyle\endcsname{%
            \addtolength\thm@preskip\parskip
            }%
        }
\endgroup

\usepackage{mathtools}
\usepackage{tikz-cd}
\usepackage{amssymb}
\usepackage{ytableau}
\usepackage{extpfeil}
\usepackage{bm}
\usepackage{xcolor}
\usepackage{mathrsfs}

\definecolor{green}{rgb}{0.,0.6,0.}

\usepackage{setspace}
\title{\vspace{-2cm}RECURRENCE RELATION FOR PLETHYSM}
\author{Étienne Tétreault}
\begin{document}
	\maketitle
	\pagenumbering{gobble}
	\pagenumbering{arabic}
	\newcommand{\s}{\mathfrak{S}}
	\newcommand{\hooklongrightarrow}{\lhook\joinrel\longrightarrow}
	\newtheorem{thm}{Theorem}[section]
	\newtheorem{prop}[thm]{Proposition}
	\newtheorem{cor}[thm]{Corollary}
	\newtheorem{conj}[thm]{Conjecture}
	\newtheorem{lem}[thm]{Lemma}
	\renewcommand{\proofname}{Proof}
	\renewcommand{\abstractname}{Abstract}
	\renewcommand{\refname}{References}

\begin{abstract}
We describe a recurrence formula for the plethysm $h_3[h_n]$. The proof is based on the original formula by Thrall.
\end{abstract}

The representation theory of the group $\text{GL}(V)$ has been developed since more than a century. We know that the irreducible representations are $S^{\nu}(V)$, where $\nu$ is a partition of an integer. When $\nu = (n)$, it is the $n^{th}$ symmetric power, and when $\nu=(1^n)$, it is the $n^{th}$ exterior product. As a representation is itself a vector space, if $\mu$ is another partition, it makes sense to consider $S^{\mu}(S^{\nu}(V))$. This is an irreducible representation of $\text{GL}(S^{\nu}(V))$, but it is not irreducible as a representation of $\text{GL}(V)$. To decompose $S^{\mu}(S^{\nu}(V))$ in irreducible representations is a key open problem in algebraic combinatorics, according to Stanley \cite{Stanley99}.

We know that the character of $S^{\nu}(V)$ is the Schur function $s_{\nu}$. They form a basis of the ring $\Lambda$ of symmetric functions. The character of $S^{\mu}(S^{\nu}(V))$ is denoted $s_{\mu}[s_{\nu}]$, and this operation is known as the \textit{plethysm} of symmetric functions, defined by Littlewood \cite{Littlewood}. Using the character theory, the decomposition of $S^{\mu}(S^{\nu}(V))$ in irreducible representations is equivalent to write the plethysm $s_{\mu}[s_{\nu}]$ as a sum of Schur functions. From the representation theory, we know that the coefficient of each Schur function appearing in the plethysm are non-negative integers. Whenever this happens, we say that the symmetric function is \textit{Schur-positive}. 

The Foulkes' conjecture goes back to 1950 \cite{Foulkes}. For a $\mathbb{C}$-vector space $V$, it states that there is an injection from $S^m(S^n(V))$ to $S^n(S^m(V))$ whenever $m \leq n$. Equivalently, as $s_{(n)} = h_n$, the $n^{th}$ homogeneous symmetric function, it states that $h_n[h_m] - h_m[h_n]$ is Schur-positive whenever $m \leq n$. This conjecture is known to hold when $n \leq 5$ \cite{Thrall}\cite{DentSiemons}\cite{McKay}\cite{CIM}, and when $n-m$ is large enough \cite{Brion}. 

The goal of this paper is to show a recurrence formula for the plethysm $h_3[h_n]$. While a general formula has been found by Thrall \cite{Thrall}, the author has not found a recurrence of this kind in the literature. This is an undergrad summer research project done at LaCIM in 2016 (Université du Québec à Montréal, Canada), under the supervision of François Bergeron.

\section{Preliminary notions}

We first need some linear operators on the ring $\mu$ of symmetric functions. As Schur functions are a basis, it suffices to describe the operators on these functions. The first one is a projection on Schur functions indexed by partitions that have at most $k$ parts, denoted $\downarrow_k$, and defined by
\[
s_{\mu} \downarrow_k = \left \{ \begin{array}{l l} s_{\mu} & \text{if} \ \ell (\mu) \leq k\\ 0 & \text{else} \\ \end{array} \right.
\]
Also, consider the following bilinear operator, denoted $\odot$, that adds the indices of two Schur functions in the following way:
\[
s_{\mu} \odot s_{\lambda} = s_{\mu + \lambda},
\]
where the sum of two partitions is done componentwise, adding zero parts if necessary.

For example, we know the following formula, that goes back to Littlewood \cite{Littlewood}:
\[
h_2[h_n] = \displaystyle \sum_{k=0}^{\lfloor \frac{n}{2} \rfloor} s_{2n-2k,2k}.
\]
Using this operator, this formula can be described recursively:
\[
h_2[h_n] = s_{22} \odot (h_2[h_{n-2}]) + s_{2n},
\]
with $h_2[h_0] = s_0 = 1$ and $h_2[h_1] = h_2$. 

We also need the original formula for $h_3[h_n]$, due to Thrall \cite{Thrall}:

\begin{thm}
For any $n \in \mathbb{N}$, we have
\[
h_3[h_n] = \displaystyle \sum_{\substack{\lambda \vdash 3n \\ \ell (\lambda) \leq 3}} f_{\lambda} s_{\lambda},
\]
where $f_{\lambda}$ can be described as follows: take $a_{\lambda} = \min\{1+\lambda_1-\lambda_2, 1+\lambda_2-\lambda_3\}$, and define $a_{\lambda}'=a_{\lambda}+i$, where $i$ is the only number in $\{-2,0,2\}$ such that $a_{\lambda}'$ is a multiple of $3$. If $a_{\lambda}'$ is even, then $f_{\lambda} = \frac{a_{\alpha}'}{6}$. If $a_{\alpha}'$ is odd and $\lambda_2$ is even, $f_{\lambda} = \frac{a_{\alpha}'+3}{6}$; if $\lambda_2$ is odd, $f_{\lambda}=\frac{a_{\alpha}'-3}{6}$. 
\end{thm}

As F. Bergeron points out to the author, we can have a recursive description of $f_{\lambda}$. For $i \in \mathbb{N}$ and $\omega \in \{0,1\}$, define
\[
g(i,\omega) = \left \{ \begin{array}{c l} g(i - 6 ,k) + 1 & \text{if} \ i \geq 6 \\ 1 & \text{if} \ k=0 \ \text{and} \ i \neq 0,2\\1 & \text{if} \ k=1 \ \text{and} \ i=4\\0 & \text{else} \\ \end{array} \right.
\] 
Then, $f_{\lambda}=g \left(a_{\lambda},\lambda_2  \mod2\right)$, with $a_{\lambda}$ as in the theorem.

\section{The formula}

We can now state and prove the following recurrence formula:

\begin{thm}
The plethysm $h_3[h_n]$ can be described as:
\begin{align*}
h_3[h_0] &= 1; \qquad \qquad h_3[h_1] = s_3; \\
(h_3[h_n])\downarrow_2 &= s_{66} \odot (h_3[h_{n-4}]) \downarrow_2 + \displaystyle \sum_{k=2}^n s_{3n-k,k} + s_{3n}; \\
h_3[h_n] &= (h_3 \circ h_n) \downarrow_2 + s_{222} \odot (h_3[h_{n-2}]) + s_{441} \odot ((h_3[h_{n-3}]) \downarrow_2 ).
\end{align*}
\end{thm}

\begin{proof}
Using Thrall formula and the bilinearity of $\odot$, we have to show that
\begin{align*}
\displaystyle \sum_{\substack{\lambda \vdash 3n \\ \ell (\lambda) \leq 2}} f_{\lambda} s_{\lambda} &= \displaystyle \sum_{\substack{\mu \vdash (3n-12) \\ \ell (\mu) \leq 2}} f_{\mu} (s_{66} \odot s_{\mu}) + \displaystyle \sum_{k=2}^n s_{3n-k,k} + s_{3n}; \\
\displaystyle \sum_{\substack{\lambda \vdash 3n \\ \ell (\lambda) = 3}} f_{\lambda} s_{\lambda}  &=  \displaystyle \sum_{\substack{\mu \vdash (3n-6) \\ \ell (\mu) \leq 3}} f_{\mu} (s_{222} \odot s_{\mu}) + \displaystyle \sum_{\substack{\mu \vdash (3n-9) \\ \ell (\mu) \leq 2}} f_{\mu} (s_{441} \odot s_{\mu}).
\end{align*}

To prove the first equation, we have to show that:
\begin{align}
f_{(3n)} &= 1; \\
 f_{(3n-1,1)} &= 0; \\
f_{(3n-k,k)} &= 1 && \text{if} \ 2 \leq k \leq 5; \\
f_{(3n-k,k)} &= f_{(3n-k-6,k-6)} + 1  && \text{if} \ 6 \leq k \leq n; \\
f_{(3n-k,k)} &= f_{(3n-k-6,k-6)} && \text{if} \ n+1 \leq k \leq \lfloor \frac{3n}{2} \rfloor.
\end{align}

$(1)$: We have $a_{(3n)} = \min\{3n+1,1\} = 1$, so $f_{(3n)} = g(1,0) = 1$. 

$(2)$: We have $a_{(3n-1,1)} =\min\{3n,2\}=2$, so $f_{(3n-1,1)} = g(2,1) = 0$. 

$(3)$: We have to consider all the cases one by one. 
\begin{enumerate}[label=\roman*)]
\item We have $a_{(3n-2,2)} =\min\{3n-1,3\}=3$, and $f_{(3n-2,2)}=g(3,0)=1$. 
\item We have $a_{(3n-3,3)} =\min\{3n-2,4\}=4$, and $f_{(3n-3,3)}=g(4,1)=1$. 
\item We have $a_{(3n-4,4)} =\min\{3n-3,5\}=5$, and $f_{(3n-4,4)}=g(5,0)=1$.
\item We have $a_{(3n-5,5)} =\min\{3n-4,6\}=6$, and $f_{(3n-5,5)}=g(6,1)=g(0,1)+1=1$.  
\end{enumerate}

$(4)$: For $6 \leq k \leq n$, we have $a_{(3n-k,k)} =\min\{1+3n-2k,1+k\}=1+k$, and $f_{(3n-k,k)} = g(1+k, k \mod 2)$. Define $\mu = (3n-k-6,k-6)$, so that $(3n-k,k) = (6,6)+\mu$. Then, we have $a_{\mu} = \min\{1+3n-2k,k-5\} = k-5$, so $f_{\mu} = g(k-5,k \mod 2)$. Applying the recursiveness of the function $g$, we have $f_{(3n-k,k)} = g(k+1,k\mod2) = g(k-5,k\mod2) + 1 = f_{\mu}+1$. 

$(5)$: For $n+1 \leq k \leq \lfloor \frac{3n}{2} \rfloor$, there are two cases. When $n+1 \leq k \leq n+2$, we have that $a_{(3n-k,k)} = \min\{1+3n-2k,1+k\}= 1+3n-2k$ and $a_{\mu} = \min\{1+3n-2k,k-5\}=k-5$. If $n+3 \leq k \leq \lfloor \frac{3n}{2} \rfloor$, $a_{(3n-k,k)} = a_{\mu} = 1+3n-2k$. So, again, we have to consider all the possible cases:
\begin{enumerate}[label=\roman*)]
\item For $k=n+1$, so $\lambda=(2n-1,n+1)$, we have $f_{\mu}=g(n-4, n-5 \mod2)$, and $f_{\lambda}=g(n-1,n+1 \mod2)$. If $m = \lfloor \frac{n-6}{6} \rfloor$ and $\ell = n \mod6$, then $n=6(m+1)+\ell$. So $f_{\mu} = g(\ell+2,n-5 \mod 2) + m$, and $f_{\lambda} = g(\ell+5,n+1\mod2)+m$. If $\omega = n+1 \mod2$, then we have the following cases:
\begin{enumerate}[label=\alph*)]
\item $\omega=0$,$\ell=1$: $f_{\mu}=g(3,0)+m=m+1$ and $f_{\lambda}=g(6,0)+m=g(0,0)+m+1=m+1$, so they are equal. 
\item $\omega=0$,$\ell=3$: $f_{\mu}=g(5,0)+m=m+1$ and $f_{\lambda}=g(8,0)+m=g(2,0)+m+1=m+1$, so they are equal. 
\item $\omega=0$,$\ell=5$: $f_{\mu}=g(7,0)+m=g(1,0)+m+1=m+2$ and $f_{\lambda}=g(10,0)+m=g(4,0)+m+1=m+2$, so they are equal. 
\item $\omega=1$,$\ell=0$: $f_{\mu}=g(2,1)+m=m$ and $f_{\lambda}=g(5,1)+m=m$, so they are equal. 
\item $\omega=1$,$\ell=2$: $f_{\mu}=g(4,1)+m=m+1$ and $f_{\lambda}=g(7,1)+m=g(1,1)+m+1=m+1$, so they are equal. 
\item $\omega=1$,$\ell=4$: $f_{\mu}=g(6,1)+m=g(0,1)+m+1=m+1$ and $f_{\lambda}=g(9,1)+m=g(3,1)+m+1=m+1$, so they are equal. 
\end{enumerate}
\item For $k=n+2$, so $\lambda=(2n-2,n+2)$, we have $f_{\mu}=g(n-3, n-4 \mod2)$, and $f_{\lambda}=g(n-3, n+2 \mod2)$, so they are equal.
\item For $n+3 \leq k \leq \lfloor \frac{3n}{2} \rfloor$, we have $f_{\mu}=g(1+3n-2k,k-5 \mod 2)$ and $f_{\lambda}=g(1+3n-2k,k+1 \mod2)$, so they are equal.
\end{enumerate}

Now, consider the second equation. For partitions $\lambda \vdash 3n$ with $3$ parts, we have to show that:
\begin{align}
f_{(\lambda_1,\lambda_2,\lambda_3)} &= f_{(\lambda_1-2,\lambda_2-2,\lambda_3-2)} && \text{if} \ \lambda_3 \geq 2; \\
f_{(\lambda_1,\lambda_2,\lambda_3)} &= f_{(\lambda_1-4,\lambda_2-4)} && \text{if} \ \lambda_3 =1 \ \text{and} \ \lambda_2 \geq 4; \\
f_{(\lambda_1,\lambda_2,\lambda_3)} &= 0 && \text{if} \ \lambda_3 =1 \ \text{and} \ \lambda_2 \leq 3.
\end{align} 

$(6)$: Define $\mu=(\lambda_1-2,\lambda_2-2,\lambda_3-2)$, so that $\lambda=(2,2,2)+\mu$. We have that $a_{\lambda}$ is either $1+\lambda_1-\lambda_2=1+(\mu_1-2)-(\mu_2+2)=1+\mu_1-\mu_2$ or $1+\lambda_2-\lambda_3=1+\mu_2-\mu_3$, so $a_{\lambda}=a_{\mu}$. We also have that $f_{\lambda}=g(a_{\lambda},\lambda_2 \mod2)$ and $f_{\mu}=g(a_{\lambda},\lambda_2-2 \mod2)$, so they are equal. 

$(7)$: If $\lambda_3=1$ and $\lambda_2 \geq 4$, then $\lambda=(3n-k-1,k,1)$ for $4 \leq k \leq \lfloor \frac{3n-1}{2}\rfloor$. Define $\mu=(3n-k-5,k-4)$, so that $\lambda = (4,4,1)+\mu$. We have to examine a few cases:
\begin{enumerate}[label=\roman*)]
\item If $4 \leq k \leq n$, then $a_{(3n-k-1,k,1)}=\min\{3n-2k,k\}=k$ and $a_{\mu}=\min\{3n-2k,k-3\}=k-3$. If $m = \lfloor \frac{k-6}{6} \rfloor$, $\ell = k \mod6$ and $\omega=k \mod2$, we have the following cases:
\begin{enumerate}[label=\alph*)]
\item $k=5$: $f_{\mu}=g(2,1)=0$ and $f_{\lambda}=g(5,1)=0$, so they are equal.
\item $\omega=0$,$\ell=0$: $f_{\mu}=g(3,0)+m=m+1$ and $f_{\lambda}=g(6,0)+m=g(0,0)+m+1=m+1$, so they are equal.
\item $\omega=0$,$\ell=2$: $f_{\mu}=g(5,0)+m=m+1$ and $f_{\lambda}=g(8,0)+m=g(2,0)+m+1=m+1$, so they are equal.
\item $\omega=0$,$\ell=4$: $f_{\mu}=g(7,0)+m=g(1,0)+m+1=m+2$ and $f_{\lambda}=g(10,0)+m=g(4,0)+m+1=m+2$, so they are equal.
\item $\omega=1$,$\ell=1$: $f_{\mu}=g(4,1)+m=m+1$ and $f_{\lambda}=g(7,1)+m=g(1,1)+m+1=m+1$, so they are equal.
\item $\omega=1$,$\ell=3$: $f_{\mu}=g(6,1)+m=g(0,1)+m+1$ and $f_{\lambda}=g(9,1)+m=g(3,1)+m+1=m+1$, so they are equal.
\item $\omega=1$,$\ell=5$: $f_{\mu}=g(8,1)+m=g(2,1)+m+1$ and $f_{\lambda}=g(11,1)+m=g(5,1)+m+1=m+1$, so they are equal.
\end{enumerate}
\item If $k=n+1$, then $\lambda=(2n-2,n+1,1)$. We have $a_{(2n-2,n+1,1)}=\min\{n-2,n+1\}=n-2$ and $a_{\mu}=\min\{n-1,n-2\}=n-2$, so they are equal. Then, $f_{\lambda} = g(a_{\lambda},k \mod2) = g(a_{\mu}, k-4 \mod2) = f_{\mu}$.
\item If $n+2 \leq k \leq \lfloor \frac{3n}{2}\rfloor$, then $a_{(3n-k-1,k,1)}=\min\{3n-2k,k\}=3n-2k$ and $a_{\mu}=\min\{3n-2k,k-3\}=3n-2k$, so they are equal. Then, $f_{(3n-k-1,k,1)}=g(a_{\lambda},k \mod2) = g(a_{\mu}, k-4 \mod2) = f_{\mu}$.
\end{enumerate}

$(8)$: If $\lambda_3=1$ and $\lambda_2 \leq 3$, then we have to test the three possibilities and see that the coefficient is zero:
\begin{enumerate}[label=\roman*)]
\item If $\lambda_2=1$, then $a_{(3n-2,1,1)}=\min\{3n-2,1\}=1$, and $f_{(3n-2,1,1)}=g(1,1)=0$.
\item If $\lambda_2=2$, then $a_{(3n-3,2,1)}=\min\{3n-4,2\}=2$, and $f_{(3n-3,2,1)}=g(2,0)=0$.
\item If $\lambda_2=3$, then $a_{(3n-4,3,1)}=\min\{3n-6,3\}=3$, and $f_{(3n-4,3,1)}=g(3,1)=0$.
\end{enumerate}

So, the formula is true.
\end{proof}

\section{Conclusion}

This recurrence gives a faster way to compute $h_3[h_n]$. The author firmly believes that such recurrence formulas can be found to compute $h_m[h_n]$ recursively for any $m$. In effect, Dent's two column result [ ? ] is equivalent to the fact that $h_m[h_n] - s_{\underbrace{22...2}_{m \ \text{times}}} \odot h_m[h_{n-2}]$ is Schur-positive. If we find other results that have a nice recursive definition, this would hint to a proof of the Foulkes' conjecture. But even in the case $h_4[h_n]$, such a recurrence is hard to find. 

\bibliographystyle{abbrv}
\nocite{*}
\bibliography{bibplethysm}

\end{document}